\documentclass[11pt]{article}
\usepackage{amsmath}
\usepackage{amsthm}
\usepackage{amssymb}
\usepackage{mathrsfs}
\usepackage{xcolor}

\usepackage{geometry}
\usepackage[colorlinks=true,
linkcolor=blue,citecolor=blue,
urlcolor=blue]{hyperref}

\makeatletter
\def\@seccntDot{.}
\def\@seccntformat#1{\csname the#1\endcsname\@seccntDot\hskip 0.5em}
\renewcommand\section{\@startsection{section}{1}{\z@}%
{18\p@ \@plus 6\p@ \@minus 3\p@}%
{9\p@ \@plus 6\p@ \@minus 3\p@}%
{\large\bfseries\boldmath}}
\renewcommand\subsection{\@startsection{subsection}{2}{\z@}%
{12\p@ \@plus 6\p@ \@minus 3\p@}%
{3\p@ \@plus 6\p@ \@minus 3\p@}%
{\bfseries\boldmath}}
\renewcommand\subsubsection{\@startsection{subsubsection}{3}{\z@}%
{12\p@ \@plus 6\p@ \@minus 3\p@}%
{\p@}%
{\bfseries\boldmath}}
\makeatother

\usepackage{microtype}

\theoremstyle{plain}
\newtheorem{theorem}{Theorem}[section]
\newtheorem{lemma}{Lemma}[section]

\newtheorem{conjecture}{Conjecture}[section]

\theoremstyle{definition}

\newtheorem{remark}{Remark}[section]

\newtheorem{claim}{Claim}

\numberwithin{equation}{section}
\allowdisplaybreaks
\newcommand{\x}{{\bf x}}

\title{Revisiting the Maximum Principal Ratio of Graphs}
 
\author{
Lele Liu\thanks{College of Science, University of Shanghai for Science and Technology, Shanghai 200093, China
(\texttt{ahhylau@outlook.com})}~,~~
Changxiang He\thanks{College of Science, University of Shanghai for Science and Technology, Shanghai 200093, China
(\texttt{changxiang-he@163.com})}
}

\date{}

\begin{document}
\maketitle

\begin{abstract}
Let $G$ be a connected graph, the principal ratio of $G$ is the ratio of the maximum and minimum entries of its
Perron eigenvector. In 2007, Cioab\v a and Gregory conjectured that among all connected graphs on $n$ vertices,
the kite graph attains the maximum principal ratio. In 2018, Tait and Tobin confirmed the conjecture for 
sufficientlty large $n$. In this article, we show the conjecture is true for all $n\geq 5000$.
\par\vspace{2mm}

\par\vspace{1mm}

\end{abstract}

\section{Introduction}
\label{sec:section-1}

In this paper, we consider only simple, undirected graphs, i.e, undirected graphs without multiple edges or loops. 
The eigenvalues of a graph $G$ are the eigenvalues $\lambda_i(G)$ of its adjacency matrix $A(G)$, and unless we 
indicate otherwise, we shall assume that $\lambda_1(G)\geq\lambda_2(G)\geq\cdots\geq\lambda_n(G)$. For a graph $G$, 
we use $N_G(v)$ and $d_G(v)$ to denote the neighborhood and degree of vertex $v$, respectively. For graph notation 
and terminology undefined here we refer the reader to \cite{BondyMurty2008}.

For a connected graph $G$, the Perron--Frobenius theorem implies that $A(G)$ has a unique positive unit eigenvector 
$\x$ corresponding to $\lambda_1(G)$, and this is called the {\em principal eigenvector} of $G$. Let $x_{\min}$ and 
$x_{\max}$ be the smallest and largest eigenvector entries, respectively. The {\em principal ratio} $\gamma(G)$ of 
$G$ is defined as
\[
\gamma(G):=\frac{x_{\max}}{x_{\min}}.
\]
Note that $\gamma(G)\geq 1$ with equality if and only if $G$ is regular. Thus, one can regard $\gamma(G)$ as a measure 
of the irregularity of $G$. 

In 1958, Schneider \cite{Schneider1958} presented an upper bound on eigenvectors of irreducible nonnegative matrices; 
for graphs it can be described as $\gamma(G)\leq (\lambda_1(G))^{n-1}$. In 2007, Nikiforov \cite{Nikiforov2007} improved 
this result for estimating the gap of spectral radius between $G$ and its proper subgraph. Latter, Cioab\v a and 
Gregory \cite{CioabGregory2007} slightly improved the previous results for $\lambda_1(G)>2$, by proving 
\[
\gamma(G) \leq \frac{\sigma^{d+1} - \sigma^{-(d+1)}}{\sigma-\sigma^{-1}},
\]
where $2\sigma=\lambda_1(G)+\sqrt{\lambda_1(G)^2-4}$, and $d$ is the shortest distance between the vertices attaining
$x_{\min}$ and $x_{\max}$. In addition, they also proved some lower bounds of $\gamma(G)$, which improved previous 
results of Ostrowski \cite{Ostrowski1952} and Zhang \cite{Zhang2005}.

Let $P_r\cdot K_s$ denote the graph of order $(r+s-1)$ attained by identifying an end vertex of a path $P_r$ to any 
vertex of a complete graph $K_s$. This graph $P_r\cdot K_s$ is called a {\em kite graph} or a {\em lollipop graph}. 
In 2007, Cioab\v a and Gregory \cite{CioabGregory2007} posed the following conjecture.

\begin{conjecture}[\cite{CioabGregory2007}]
Among all connected graphs on $n$ vertices, the kite graph attains the maximum principal ratio.
\end{conjecture}

In 2018, Tait and Tobin \cite{TaitTobin2018} confirmed this conjecture for sufficientlty large $n$. 

\begin{theorem}[\cite{TaitTobin2018}]
For sufficiently large $n$, the connected graph $G$ on $n$ vertices with largest principal ratio is a kite graph.
\end{theorem}

In this article, we prove that the conjecture is true for all $n\geq 5000$.

\begin{theorem}\label{thm:Main-result}
Among all connected graphs on $n\geq 5000$ vertices, the kite graph attains the maximum principal ratio.
\end{theorem}

\section{Preliminaries}

Throughout this paper $G$ will be a connected simple graph on $n$ vertices. Let $\x$ be the principal eigenvector of 
$G$, and $x_{v_i}$ be the entry of $\x$ corresponding to $v_i\in V(G)$. For convenience, we always scale the principal 
eigenvector $\x$ of $G$ so that its maximum entry is $1$. 

As mentioned in Section \ref{sec:section-1}, for a graph $G$ with $\lambda_1(G)\geq 2$ we denote
\begin{equation}\label{eq:sigma-G}
\sigma(G)=\frac{\lambda_1(G)+\sqrt{\lambda_1(G)^2-4}}{2}.
\end{equation}
Recall that a {\em pendant path} is a path with one of its end vertices having degree one and all the internal 
vertices having degree two. 

The following lemma give an upper bound for principal ratio $\gamma(G)$ in terms of $\sigma(G)$.

\begin{lemma}[\cite{TaitTobin2018}]\label{lem:gamma-upper-bound}
Suppose that $v_1$ and $v_k$ are the vertices with minimum and maximum eigenvector entries, respectively. 
Let $\lambda:=\lambda_1(G) > 2$ and $\sigma:=\sigma(G)$, then for $1\leq j\leq k$, 
\[
\gamma(G)\leq\frac{\sigma^j-\sigma^{-j}}{\sigma-\sigma^{-1}}\cdot \frac{1}{x_{v_j}}.
\]
Moreover, equality holds if the vertices $v_1$,$v_2$,$\ldots$, $v_j$ form a pendant path.
\end{lemma}

\begin{lemma}\label{lem:exact-sigma-bound}
For a connected graph $G$, let $\lambda:=\lambda_1(G) > 2$ and $\sigma:=\sigma(G)$. If $j\geq 2$, then 
\begin{equation}\label{eq:exact-sigma-bound}
\Big(\lambda - \frac{1}{\lambda - 1}\Big)^{j-2}\lambda
\leq \frac{\sigma^j-\sigma^{-j}}{\sigma - \sigma^{-1}}
\leq \Big(\lambda - \frac{1}{\lambda}\Big)^{j-2}\lambda.
\end{equation}
\end{lemma}

\begin{proof}
If $j=2$, the assertion holds trivially, so we assume that $j\geq 3$. We start with the proof of the left-hand 
side of \eqref{eq:exact-sigma-bound}. In view of \eqref{eq:sigma-G}, $\lambda=\sigma+\sigma^{-1}$. Let 
$i\geq 3$, then 
\begin{align*}
\frac{\sigma^i-\sigma^{-i}}{\sigma^{i-1}-\sigma^{-(i-1)}}
-\Big(\lambda-\frac{1}{\lambda-1}\Big)
& = \frac{\sigma^i-\sigma^{-i}}{\sigma^{i-1}-\sigma^{-(i-1)}}
-\Big((\sigma+\sigma^{-1})-\frac{1}{\sigma+\sigma^{-1}-1}\Big) \\
& = \frac{\sigma^{i-3}}{\sigma^{i-1} - \sigma^{-(i-1)}}\cdot\frac{1}{\lambda-1}
(\sigma-1+\sigma^{-2i+6}-\sigma^{-2i+5}) \\
& >0,
\end{align*}
which implies that
\[
\frac{\sigma^i-\sigma^{-i}}{\sigma^{i-1}-\sigma^{-(i-1)}} 
> \lambda-\frac{1}{\lambda-1},~~ i\geq 3.
\]
Multiplying these inequalities for $i=3,\ldots,j$ and noting that $\sigma+\sigma^{-1}=\lambda$, we obtain 
the desired inequality.

Next we shall prove the right-hand side. Likewise, for $i\geq 3$ we see that
\[ 
\frac{\sigma^i-\sigma^{-i}}{\sigma^{i-1}-\sigma^{-(i-1)}}
- \Big(\lambda-\frac{1}{\lambda}\Big)
= \frac{\sigma^{i-3}}{\lambda (\sigma^{i-1} - \sigma^{-(i-1)})} (\sigma^{-2i+6} -1)
< 0.
\]
The inequality follows by analogous arguments as above.
\end{proof}

\begin{lemma}\label{lem:f-Monotonicity}
Let $n\geq e^e$, the function $f_n(x)=(n-x)^x$ is increasing in $x$ whenever
$1 < x < n-\frac{n}{\log n} \big(1+\frac{1}{\sqrt{\log n}}\big)$, and decreasing 
in $x$ whenever $n-\frac{n}{\log n}\big(1+\frac{1}{\log n}\big) < x < n$.
\end{lemma}

\begin{proof}
Taking the derivative of $f_n(x)$ we find that
\[
\frac{f_n'(x)}{f_n(x)}=\log (n-x)-\frac{n}{n-x}+1: =g(x).
\] 
Obviously, $g(x)$ is decreasing in $x$. By some algebra (details can be found in Appendix \ref{appendix:g-vs-zero}), 
we have 
\begin{equation}\label{eq:g-vs-zero}
g\Big(n-\frac{n}{\log n} \Big(1+\frac{1}{\sqrt{\log n}}\Big)\Big) > 0, \ \text{and}\
g\Big(n-\frac{n}{\log n}\Big(1+\frac{1}{\log n}\Big)\Big) < 0. 
\end{equation}
The results follow from the monotonicity of $g(x)$.
\end{proof}

\section{Proof of \autoref{thm:Main-result}}

Let $G^*$ be a graph with maximum principal ratio among all connected graphs on $n$ vertices. Let $\x$ be the 
principal eigenvector of $G^*$ with maximum entry is $1$, and $(k-1)$ be the shortest distance between a vertex 
on which $\x$ component is minimum and a vertex on which it is maximum. Suppose that $x_{v_1}=x_{\min}$ and 
$x_{v_k}=x_{\max}$, and $v_1,v_2,\ldots,v_k$ is a path between $v_1$ and $v_k$. 

Throughout this paper, we set $\lambda:=\lambda_1(G^*)$, $\sigma:=\sigma(G^*)$ and $N(v):=N_{G^*}(v)$; and write 
$x_i:=x_{v_i}$ for brevity. Let us note that most of our results are about $n\geq 5000$, so we tacitly assume 
that $n\geq 5000$.

We need the following facts, which are Lemma 3.2 and Lemma 3.6 in \cite{TaitTobin2018}.

\begin{lemma}[\cite{TaitTobin2018}]\label{lem:pendant-path}
The vertices $v_1,v_2,\ldots,v_{k-1}$ form a pendant path in $G^*$, and $v_k$ is connected to each vertex in $G^*$ 
that is not on this path.    
\end{lemma}

\begin{remark}
By \autoref{lem:pendant-path}, $d_{G^*}(v_k)=n-k+1$. This implies that $d_{G^*}(v)\leq n-k+1$ for each $v\in V(G^*)\setminus\{v_1,\ldots,v_k\}$.
Since $x_k=x_{\max}$, we see $d_{G^*}(v_{k-1})\leq n-k+1$. Thus, the maximum degree of $G^*$ is $n-k+1$, and therefore
$\lambda<n-k+1$.
\end{remark}

As for the lower bound of $\lambda$, Tait and Tobin \cite{TaitTobin2018} proved the following result.

\begin{lemma}[\cite{TaitTobin2018}]
$\lambda > n-k$.
\end{lemma}

Next, we shall show that the degree of $v_{k-1}$ is exactly $2$ in $G^*$, which together with \autoref{lem:pendant-path} 
implies that $v_1,v_2,\ldots,v_k$ is a pendant path. With this in hand, we prove that $V(G^*)\setminus\{v_1,\ldots,v_k\}$
induced a clique, and therefore completing the proof of \autoref{thm:Main-result}.

To prove our main result, we need make an estimate of $k$, as stated in the following lemma.

\begin{lemma}\label{lem:bound-k}
Let $n\geq 5000$, then
\begin{equation}\label{eq:bound-k}
n-\frac{n}{\log n} \Big(1+\frac{1.1}{\sqrt{\log n}}\Big) < k < n-\frac{n}{\log n}\Big(1-\frac{1}{\log n}\Big).
\end{equation}
\end{lemma}

\begin{proof}
By \autoref{lem:gamma-upper-bound} and \autoref{lem:exact-sigma-bound}, we get 
\[
\gamma(G^*)\leq\frac{\sigma^k-\sigma^{-k}}{\sigma-\sigma^{-1}}\leq\lambda^{k-1}<(n-k+1)^{k-1}.
\]
On the other hand, letting $H=P_j\cdot K_{n-j+1}$, we have 
\[
\gamma(H)=\frac{\sigma(H)^j-\sigma(H)^{-j}}{\sigma(H)-\sigma(H)^{-1}}>(\lambda_1(H)-1)^{j-1}>(n-j-1)^{j-1}.
\]
Combining these two inequalities we immediately obtain 
\begin{equation}\label{eq:inequality-for-k}
(n-k+1)^{k-1}>(n-j-1)^{j-1}.
\end{equation}

For the left-hand side of \eqref{eq:bound-k}, let $j=n-1-\big\lceil \frac{n}{\log n}\big(1+\frac{1}{\sqrt{\log n}}\big)\big\rceil$.
Since 
\[
n-2-\frac{n}{\log n}\Big(1+\frac{1}{\sqrt{\log n}}\Big) < j < n-1-\frac{n-2}{\log (n-2)} \bigg(1+\frac{1}{\sqrt{\log (n-2)}}\bigg),
\] 
and $f_{n-2}(x-1)$ is increasing in $x$ whenever $2 < x < n-1-\frac{n-2}{\log (n-2)} \big(1+\frac{1}{\sqrt{\log (n-2)}}\big)$ 
by \autoref{lem:f-Monotonicity}, we deduce that
\begin{align*}
(n-j-1)^{j-1} 
& = f_{n-2}(j-1) > f_{n-2} \Big(n-3-\frac{n}{\log n}\Big(1+\frac{1}{\sqrt{\log n}}\Big)\Big) \\
& = \Big(1+\frac{n}{\log n}\Big(1+\frac{1}{\sqrt{\log n}}\Big)\Big)^{n-3-\frac{n}{\log n}\big(1+\frac{1}{\sqrt{\log n}}\big)}.
\end{align*}
If $k\leq n-\frac{n}{\log n} \big(1+\frac{1.1}{\sqrt{\log n}}\big)$, then 
\begin{align*}
(n-k+1)^{k-1} 
& = f_n(k-1)\leq f_n\Big(n-1-\frac{n}{\log n} \Big(1+\frac{1.1}{\sqrt{\log n}}\Big)\Big) \\
& = \Big(1+\frac{n}{\log n} \Big(1+\frac{1.1}{\sqrt{\log n}}\Big)\Big)^{n-1-\frac{n}{\log n} \big(1+\frac{1.1}{\sqrt{\log n}}\big)}. 
\end{align*}
By some involved calculations (details can be found in Appendix \ref{appendix:k-lower-bound}) one can show that
\begin{equation}\label{eq:k-lower-bound}
f_n\Big(n-1-\frac{n}{\log n} \Big(1+\frac{1.1}{\sqrt{\log n}}\Big)\Big)
< f_{n-2} \Big(n-3-\frac{n}{\log n}\Big(1+\frac{1}{\sqrt{\log n}}\Big)\Big).
\end{equation}
Therefore, $(n-k+1)^{k-1} < (n-j-1)^{j-1}$, contrary to \eqref{eq:inequality-for-k}. This proves the left-hand side of \eqref{eq:bound-k}.

For the right-hand side of \eqref{eq:bound-k}, let $j=n+1-\big\lceil \frac{n}{\log n}\big(1+\frac{1}{\log n}\big) \big\rceil$, 
then
\[
j > n-\frac{n}{\log n}\Big(1+\frac{1}{\log n} \Big) 
> n-1-\frac{n-2}{\log (n-2)}\Big(1+\frac{1}{\log (n-2)}\Big).
\]
By \autoref{lem:f-Monotonicity}, $f_{n-2}(x-1)$ is decreasing in $x$ if $n-1-\frac{n-2}{\log (n-2)} \big(1+\frac{1}{\log (n-2)}\big) < x < n-1$. 
Hence,
\[
(n-j-1)^{j-1} = f_{n-2}(j-1) > 
f_{n-2}\Big(n-\frac{n}{\log n}\Big(1+\frac{1}{\log n} \Big) \Big).
\]
If $k\geq n-\frac{n}{\log n}(1-\frac{1}{\log n})$, then 
\[ 
(n-k+1)^{k-1} = f_n(k-1) \leq f_n\Big(n-1-\frac{n}{\log n} \Big(1-\frac{1}{\log n}\Big)\Big).
\]
By some algebra (see Appendix \ref{appendix:k-upper-bound}) we obtain
\begin{equation}\label{eq:k-upper-bound}
f_n\Big(n-1-\frac{n}{\log n} \Big(1-\frac{1}{\log n}\Big)\Big) 
< f_{n-2}\Big(n-\frac{n}{\log n}\Big(1+\frac{1}{\log n} \Big) \Big), 
\end{equation}
and thus $(n-k+1)^{k-1} < (n-j-1)^{j-1}$, which is a contradiction to \eqref{eq:inequality-for-k}. This completes the proof.
\end{proof}

\begin{lemma}\label{lem:normal-x}
$\lambda < \|\x\|_2^2 < \lambda + 10/9$.
\end{lemma}

\begin{proof}
The left-hand side follows from Cauchy\,--\,Schwarz inequality. Indeed,
\[
\|\x\|_2^2 \geq x_k^2 + \sum_{v\in N(v_k)} x_v^2 
\geq 1 + \frac{\lambda^2}{n-k+1} 
> 1 + \frac{\lambda^2}{\lambda+1}
> \lambda.
\]

To prove the right-hand side, letting $x_0=0$ and noting that $x_{i-1}<x_i$ for $1\leq i\leq k-1$, we have
\[
\lambda x_i = x_{i-1} + x_{i+1} < x_i + x_{i+1},
\]
implying that $x_i<x_{i+1}/(\lambda-1)$. Hence,
\[
x_i < \frac{x_{k-1}}{(\lambda -1)^{k-1-i}} < 
\frac{1}{(\lambda -1)^{k-1-i}},~~i=1,2,\ldots,k-1.
\]
It follows that
\begin{align*}
\|\x\|_2^2 
& = \sum_{u\in N(v_k)} x_u^2 + \sum_{i=1}^{k-2} x_i^2 + x_k^2 \\
& < \sum_{u\in N(v_k)} x_u + \sum_{i=1}^{k-1} \frac{1}{(\lambda -1)^{k-1-i}} \\
& < \lambda +1 + \frac{1}{\lambda-2} \\
& < \lambda +\frac{10}{9},
\end{align*}
completing the proof.
\end{proof}

Hereafter, we use $S$ to denote the set $V(G^*)\setminus\{v_1,\ldots,v_k\}$, and write $U:= S\cap N(v_{k-1})$.
The next lemma gives a more precise upper bound for $\lambda$.

\begin{lemma}\label{lem:small-lambda}
$\lambda < n-k+3/5$.
\end{lemma}

\begin{proof}
Assume for contradiction that $\lambda\geq n-k+3/5$, we have the following claims.

\begin{claim}\label{claim:xu-large}
$x_v\geq 3/5$ for each $v\in N(v_k)$.
\end{claim} 

\noindent {\it Proof of \autoref{claim:xu-large}}. If there is $w\in N(v_k)$ such that $x_w< 3/5$, then 
\[
n-k+\frac{3}{5}\leq\lambda x_k = \sum_{u\in N(v_k)} x_u < \frac{3}{5}+(n-k),
\]
a contradiction completing the proof of the claim.

\begin{claim}\label{claim}
There is $z\in S$ such that $v_{k-1}z\notin E(G^*)$. 
\end{claim}

\noindent {\it Proof of \autoref{claim}}. 
Assume for contradiction that $v_{k-1}v\in E(G^*)$ for each $v\in S$. From the eigenvalue equations for $v_{k-1}$ 
and $v_k$, we get $\lambda (x_{k-1} - x_k) = x_k-x_{k-1}+x_{k-2}>0$, and therefore $x_{k-1} > x_k = 1$, which leads 
to a contradiction. This completes the proof of the claim.
\par\vspace{2mm}

Now, let $G^{+} = G^* + \{v_{k-1}z\}$, and $\x^+$ be the principal eigenvector of $G^*$ with maximum entry $1$.
Our goal is to show $\gamma(G^+)>\gamma(G^*)$, and therefore deduce 
a contradiction.

Set $\lambda^+:=\lambda_1(G^+)$ and $\sigma^+:=\sigma(G^+)$ for short. By the Rayleigh principle and \autoref{claim:xu-large} 
we obtain
\[
\lambda^+-\lambda \geq \frac{\x^{\mathrm{T}} A(G^+)\x-\x^{\mathrm{T}} A(G^*)\x}{\|\x\|_2^2}
=\frac{2x_{k-1} x_z}{\|\x\|_2^2}>\frac{18}{25\|\x\|_2^2},
\]
which, together with \autoref{lem:normal-x}, gives
\begin{equation}\label{eq:lambda-plus-lambda}
\lambda^+ > \lambda + \frac{18}{25(\lambda+10/9)}.
\end{equation}

The next claim gives an upper bound for $x_{k-1}^+ - x_{k-1}$.

\begin{claim}\label{claim:xplus-x}
$x_{k-1}^+ - x_{k-1} < 1.5/\lambda$.
\end{claim}

\noindent {\it Proof of \autoref{claim:xplus-x}.} By the eigenvalue equations for $\lambda$ and $\lambda^+$ with respect 
to $v_{k-1}$, we have
\[
x_{k-1}=\frac{1}{\lambda} \Big(x_{k-2}+x_k+\sum_{v\in U} x_v\Big),~~
x_{k-1}^+=\frac{1}{\lambda^+} \Big(x_{k-2}^+ + x_k^+ + x_z^+ +\sum_{v\in U} x_v^+\Big).
\]
It follows from $\lambda^+>\lambda$ that
\[
x_{k-1}^+ - x_{k-1} <\frac{1}{\lambda} \Big( (x_z^+ -x_{k-2}) + (x_k^+ -x_k) + x_{k-2}^+ +\sum_{v\in U} (x_v^+ - x_v)\Big).
\]
Note that $x_z^+-x_{k-2}<1$, $x_k^+-x_k\leq 0$ and $x_{k-2}^+<1/(\lambda-1)$. Hence it remains to bound $\sum_{v\in U} (x_v^+ - x_v)$.
To this end, noting that
\begin{equation}\label{eq:sum-x-vs-U}
\sum_{v\in N(v_k)\setminus U} x_v\leq d_{G^*}(v_k)-|U|, \ \text{and} \
\sum_{v\in N(v_k)} x_v=\lambda\geq d_{G^*}(v_k)-\frac{2}{5},
\end{equation}
we immediately obtain 
\begin{equation}\label{eq:sum-x-vs-U-2}
|U|-\frac{2}{5} \leq\sum_{v\in U} x_v \leq |U|.
\end{equation}
Likewise, using the same argument as in \eqref{eq:sum-x-vs-U}, we derive that
\[
|U|-\frac{2}{5} \leq\sum_{v\in U} x_v^+ \leq |U|,
\]
which, together with \eqref{eq:sum-x-vs-U-2}, yields that
\[
\frac{2}{5} \leq \sum_{v\in U} (x_v^+ - x_v) \leq \frac{2}{5}.
\]
Combining the above inequalities, we arrive at
\[
x_{k-1}^+ - x_{k-1} < \frac{1+(\lambda -1)^{-1} + 2/5}{\lambda}
< \frac{3}{2\lambda}.
\]
This completes the proof of the claim.
\vspace{3mm}

To compare $\gamma(G^+)$ with $\gamma(G^*)$, we use \autoref{lem:gamma-upper-bound} and \autoref{lem:exact-sigma-bound} 
to bound them. On the one hand,
\[
\gamma(G^+) 
= \frac{(\sigma^+)^{k-1}-(\sigma^+)^{-(k-1)}}{\sigma^+ - (\sigma^+)^{-1}}\cdot\frac{1}{x_{k-1}^+}
> \bigg(\lambda^+-\frac{1}{\lambda^+-1}\bigg)^{k-3} \cdot\frac{\lambda^+}{x_{k-1}^+}.
\]
On the other hand, we have
\[
\gamma(G^*) 
= \frac{\sigma^{(k-1)}-\sigma^{-(k-1)}}{\sigma - \sigma^{-1}}\cdot\frac{1}{x_{k-1}}
< \Big(\lambda-\frac{1}{\lambda}\Big)^{k-3} \cdot\frac{\lambda}{x_{k-1}}.
\]
Setting $\alpha:=\frac{18}{25(\lambda+10/9)}$, and noting that \eqref{eq:lambda-plus-lambda} we have
\begin{align*}
\frac{\gamma(G^+)}{\gamma(G^*)} 
& > \bigg(\frac{\lambda^+ -(\lambda^+-1)^{-1}}{\lambda-\lambda^{-1}}\bigg)^{k-3} \cdot\frac{x_{k-1}}{x^+_{k-1}} \\
& > \bigg(\frac{\lambda+\alpha-(\lambda-1)^{-1}}{\lambda-\lambda^{-1}}\bigg)^{k-3} \cdot\frac{x_{k-1}}{x^+_{k-1}} \\
& = \bigg(1+\frac{\alpha\lambda-(\lambda-1)^{-1}}{\lambda^2-1}\bigg)^{k-3} \cdot\frac{x_{k-1}}{x^+_{k-1}}.
\end{align*}
By simple algebra, we find that
\[
\alpha\lambda-\frac{1}{\lambda-1} = \frac{18}{25}-\frac{4}{5(\lambda+10/9)} - \frac{1}{\lambda-1},
\]
and $\lambda>51$. Consequently,
\[
\alpha\lambda-\frac{1}{\lambda-1} > \frac{17}{25}.
\]
It follows from Bernoulli's inequality that
\[ 
\frac{\gamma(G^+)}{\gamma(G^*)}
> \bigg(1+\frac{17(k-3)}{25(\lambda^2-1)}\bigg) \cdot\frac{x_{k-1}}{x^+_{k-1}} 
> \bigg(1+\frac{16k}{25\lambda^2}\bigg) \cdot\frac{x_{k-1}}{x^+_{k-1}}.
\]
By \autoref{lem:bound-k} and $\lambda < n-k+1$, we see $k/\lambda > (\log n)/2$. Indeed,
\[ 
\frac{k}{\lambda} > \frac{k}{n+1-k} 
> \frac{\log n - \big(1+\frac{1.1}{\sqrt{\log n}}\big)}{\frac{1.1}{\sqrt{\log n}} + \frac{\log n}{n} + 1} 
> \frac{\log n - 3/2}{3/2} 
> \frac{\log n}{2},
\]
where the third inequality due to the fact that $1.1 \cdot(\log n)^{-1/2} + (\log n)/n < 1/2$. As a consequence,
\begin{equation}\label{eq:ratio-Gplus-G}
\frac{\gamma(G^+)}{\gamma(G^*)} 
> \Big(1+\frac{8\cdot\log n}{25\lambda}\Big) \cdot\frac{x_{k-1}}{x^+_{k-1}}.
\end{equation}

To finish the proof, we consider the following two cases:

\noindent {\bfseries Case 1:} The maximum eigenvector entry of $\x^+$ is still attained by vertex $v_k$. 

In light of \autoref{claim:xplus-x} and $x_{k-1}^+>3/5$, we deduce that 
\[
\frac{x_{k-1}}{x_{k-1}^+}>1-\frac{3}{2\lambda x_{k-1}^+}>1-\frac{5}{2\lambda},
\]
which, together with \eqref{eq:ratio-Gplus-G}, gives
\[
\frac{\gamma(G^+)}{\gamma(G^*)} 
> \Big(1+\frac{8\cdot\log n}{25\lambda}\Big) \Big(1-\frac{5}{2\lambda}\Big)
> 1. 
\]
\noindent {\bfseries Case 2:} The maximum eigenvector entry of $\x^+$ is no longer attained by vertex $v_k$. 

For any vertex $v\in S$, by the eigenvalue equations for $v_k$ and $v$, we get
\[
\lambda^+ x_k^+ = x_{k-1}^+ + \sum_{u\in S} x_u^+,\ \text{and}\
\lambda^+ x_v^+ \leq x_{k-1}^+ + x_k^+ + \sum_{u\in S\setminus\{v\}} x_u^+,
\]
which implies that $(\lambda^+ + 1) (x_k^+ - x_v^+) \geq 0$, and thus $x_k^+ \geq x_v^+$. Hence the maximum 
eigenvector entry of $\x^+$ must be attained by $v_{k-1}$. Therefore,
\[
\gamma(G^+) = \frac{(\sigma^+)^{k-1} - (\sigma^+)^{-(k-1)}}{\sigma^+ - (\sigma^+)^{-1}}.
\]
Applying \autoref{claim:xplus-x} to $x_{k-1}^+=1$, we see $x_{k-1} > 1-1.5/\lambda$. By \eqref{eq:ratio-Gplus-G} 
again, we have
\[
\frac{\gamma(G^+)}{\gamma(G^*)} > \Big(1+\frac{8\cdot\log n}{25\lambda}\Big)\cdot x_{k-1}
>1.
\]

Summing the above two cases, we see $\gamma(G^+)>\gamma(G^*)$, which is a contradiction to the maximality 
of $\gamma(G^*)$.
\end{proof}

\begin{lemma}\label{lem:xk-1-upper-bound}
$x_{k-1} < n^{-0.24}$.
\end{lemma}

\begin{proof}
Let $H=P_{k-1}\cdot K_{n-k+2}$. In view of \autoref{lem:gamma-upper-bound} and \autoref{lem:exact-sigma-bound} we 
conclude that
\begin{align*}
\gamma(H) 
& = \frac{\sigma(H)^{k-1}-\sigma(H)^{-(k-1)}}{\sigma(H)-\sigma(H)^{-1}} \\
& \geq \Big(\lambda_1(H)-\frac{1}{\lambda_1(H)-1}\Big)^{k-3}\cdot\lambda_1(H) \\
& > \Big(n-k+1-\frac{1}{n-k}\Big)^{k-3} (n-k+1).
\end{align*}
On the other hand, from \autoref{lem:small-lambda} we have
\[
\gamma(G^*) = \frac{\sigma^{k-1}-\sigma^{-(k-1)}}{\sigma-\sigma^{-1}}\cdot\frac{1}{x_{k-1}}
< \Big(n-k+\frac{3}{5}-\frac{1}{n-k+3/5}\Big)^{k-3} \cdot\frac{n-k+3/5}{x_{k-1}}.
\]
Since $\gamma(G^*)\geq\gamma(H)$ we deduce that
\begin{align*}
x_{k-1} 
& < \bigg(\frac{n-k+3/5-(n-k+3/5)^{-1}}{n-k+1-(n-k)^{-1}}\bigg)^{k-3}\cdot\frac{n-k+3/5}{n-k+1} \\
& = \bigg(1-\frac{2/5+(n-k+3/5)^{-1}-(n-k)^{-1}}{n-k+1-(n-k)^{-1}}\bigg)^{k-3}\cdot\frac{n-k+3/5}{n-k+1} \\
& < \Big(1-\frac{2}{5(n-k+2)}\Big)^{k-3} \\
& < e^{-\frac{0.4(k-3)}{n-k+2}}.
\end{align*}

In what follows, we shall give a lower bound for $\frac{k-3}{n-k+2}$. To this end, note that 
\[
\frac{1.1}{\sqrt{\log n}}+\frac{3\cdot\log n}{n}+1 < \frac{7}{5}.
\]
By \autoref{lem:bound-k}, we obtain 
\[ 
\frac{k-3}{n-k+2} > \frac{\log n-\big(\frac{1.1}{\sqrt{\log n}}+\frac{3\cdot\log n}{n}+1\big)}{\frac{1.1}{\sqrt{\log n}}+\frac{2\cdot\log n}{n}+1}
> \frac{\log n-7/5}{7/5} > \frac{3\cdot\log n}{5}.
\]
Finally, we get
\[
x_{k-1} < e^{-0.24\cdot\log n} = n^{-0.24},
\]
completing the proof.
\end{proof}

\begin{lemma}\label{lem:degree-vk-1-two}
The degree of $v_{k-1}$ is $2$ in $G^*$.
\end{lemma}

\begin{proof}
It suffices to show that $|U|=0$. Assume for contradiction that $|U|>0$. Let $G^-$ be the graph obtained from $G^*$ by 
removing these $|U|$ edges, i.e., $G^- = G-\{v_{k-1}v: v\in U\}$. 

Our goal is to show that $\gamma(G^-)>\gamma(G^*)$, and therefore get a contradiction. To this end, we first show that 
$x_u>1-x_{k-1}$ for each $u\in S$. Indeed, if there is $w\in S$ such that $x_w\leq 1-x_{k-1}$, then 
\[
n-k < \lambda x_k = x_{k-1} + \sum_{u\in S} x_u \leq x_{k-1} + x_w + (n-k-1) \leq n-k,
\]
a contradiction yielding $x_u>1-x_{k-1}$ for $u\in S$.

By eigenvalue equation we see
\[
\lambda x_{k-1} > \sum_{u\in U} x_u > |U| (1-x_{k-1}),
\]
and hence
\begin{equation}\label{eq:size-U}
|U| < \frac{\lambda x_{k-1}}{1-x_{k-1}}.
\end{equation}
Combining \eqref{eq:size-U}, the Rayleigh principle, \autoref{lem:normal-x} and \autoref{lem:xk-1-upper-bound}, we deduce that
\begin{equation}\label{eq:lambda-lambda-minus}
\lambda-\lambda^-\leq\frac{2|U|x_{k-1}}{\|\x\|^2_2} < \frac{2|U|x_{k-1}}{\lambda}
< \frac{2x_{k-1}^2}{1-x_{k-1}} < \frac{2}{n^{0.46}}. 
\end{equation}
Using the eigenvalue equation for $\lambda$ and $v_{k-1}$, and \autoref{lem:xk-1-upper-bound}, we obtain
\[
\lambda x_{k-1} > 1+ \sum_{u\in U} x_u > 1+|U|(1-x_{k-1})
> 2-n^{-0.24} > 1.87,
\]
which implies that 
\begin{equation}\label{eq:xk-1-lower-bound}
x_{k-1} > \frac{1.87}{\lambda}.
\end{equation}

Now we are ready to compare $\gamma(G^-)$ with $\gamma(G^*)$. By \autoref{lem:exact-sigma-bound} and 
\eqref{eq:xk-1-lower-bound} we see 
\[
\gamma(G^*)=\frac{\sigma^{k-1}-\sigma^{-(k-1)}}{\sigma-\sigma^{-1}}\cdot\frac{1}{x_{k-1}}
<\frac{\lambda^2}{1.87}\Big(\lambda-\frac{1}{\lambda}\Big)^{k-3}.
\]
On the other hand,
\[
\gamma(G^-)=\frac{(\sigma^-)^k-(\sigma^-)^{-k}}{\sigma^--(\sigma^-)^{-1}}
>\Big(\lambda^--\frac{1}{\lambda^- -1}\Big)^{k-1}
\]
For short, denote $\beta:=2 n^{-0.46}$. Hence, the above two inequalities, together with \eqref{eq:lambda-lambda-minus}, 
imply that
\begin{align*}
\frac{\gamma(G^-)}{\gamma(G^*)} 
& > 1.87 \bigg(\frac{\lambda^- - (\lambda^- -1)^{-1}}{\lambda - \lambda^{-1}}\bigg)^{k-3}
\bigg(\frac{\lambda^- - (\lambda^- -1)^{-1}}{\lambda}\bigg)^2 \\
& = 1.87 \bigg(1-\frac{(\lambda-\lambda^{-})+ (\lambda^- -1)^{-1} - \lambda^{-1}}{\lambda - \lambda^{-1}}\bigg)^{k-3} 
\bigg(1-\frac{(\lambda-\lambda^-)+ (\lambda^- -1)^{-1}}{\lambda}\bigg)^2 \\
& > 1.87 \bigg(1-\frac{\beta+ (\lambda -2)^{-1} - \lambda^{-1}}{\lambda - \lambda^{-1}}\bigg)^{k-3} 
\bigg(1-\frac{\beta + (\lambda-2)^{-1}}{\lambda}\bigg)^2.
\end{align*}
Using \autoref{lem:bound-k} and the fact $\lambda >n-k$, we deduce that $\lambda>500$. Then
\[
1-\frac{\beta + (\lambda-2)^{-1}}{\lambda} > 1-\frac{1}{250} > 0.99.
\]
By some algebra we can also obtain 
\[
\frac{\beta+ (\lambda -2)^{-1} - \lambda^{-1}}{\lambda - \lambda^{-1}} < \frac{1.1\beta}{\lambda}.
\] 
Combining these inequalities above we have
\begin{equation}\label{eq:gamma-minus-over-gamma-star}
\frac{\gamma(G^-)}{\gamma(G^*)} 
> 1.87\times (0.99)^2\times \Big(1-\frac{1.1\beta}{\lambda}\Big)^{k-3} 
> 1.83\times \Big(1-\frac{2.2(k-3)}{\lambda n^{0.46}}\Big).
\end{equation}
Finally, it remains to bound $(k-3)/\lambda$. By some calculations, we get
\[
\frac{k-3}{\lambda} < \frac{k-3}{n-k} < \frac{25\cdot\log n}{22},
\]
which, together with \eqref{eq:gamma-minus-over-gamma-star}, gives
\[
\frac{\gamma(G^-)}{\gamma(G^*)} > 1.83\times\Big(1-\frac{5}{2}\cdot\frac{\log n}{n^{0.46}}\Big) >1.
\]
This contradiction completing the proof of \autoref{lem:degree-vk-1-two}.
\end{proof}

We are now in the position to prove the main result of this paper. 
\vspace{3mm}

\noindent {\bfseries Proof of \autoref{thm:Main-result}}.
By \autoref{lem:pendant-path} and \autoref{lem:degree-vk-1-two}, it suffices to show that the induced subgraph $G^*[S]$ 
is a complete graph. If not, then $G^*$ is a proper subgraph of $H:=P_k\cdot K_{n-k+1}$, and hence $\lambda_1(H)>\lambda_1(G^*)$.

In view of \autoref{lem:gamma-upper-bound} we get 
\[
\gamma(H) = \frac{\sigma(H)^k - \sigma(H)^{-k}}{\sigma(H) - \sigma(H)^{-1}},~~
\gamma(G^*) = \frac{\sigma^k - \sigma^{-k}}{\sigma - \sigma^{-1}}.
\]
Since the functions $p(x)=(x^k-x^{-k})/(x-x^{-1})$ and $q(x)=x+\sqrt{x^2-4}$ are both increasing whenever $x>2$,
we obtain $\gamma(H)>\gamma(G^*)$, a contradiction completing the proof of \autoref{thm:Main-result}.

\newpage
\appendix

\begin{center}
\bfseries\large Appendix
\end{center} 

\section{Proof of Inequality \eqref{eq:g-vs-zero}}
\label{appendix:g-vs-zero}

Let $g(x)=\log (n-x)-\frac{n}{n-x}+1$, then
\begin{align*}
g\Big(n-\frac{n}{\log n} \Big(1+\frac{1}{\sqrt{\log n}}\Big)\Big) 
& = \log\Big(\frac{n}{\log n}\Big(1+\frac{1}{\sqrt{\log n}}\Big)\Big) - \frac{\log n}{1+(\log n)^{-1/2}} +1 \\
& > \log\Big(\frac{n}{\log n}\Big) - \frac{\log n}{1+(\log n)^{-1/2}} + 1.
\end{align*}
Since $(1+x)^{-1} < 1-x/2$ for $0< x<1$, we obtain
\begin{align*}
g\Big(n-\frac{n}{\log n} \Big(1+\frac{1}{\sqrt{\log n}}\Big)\Big) 
& > \log\Big(\frac{n}{\log n}\Big) - \Big(1-\frac{1}{2\sqrt{\log n}}\Big) \log n + 1 \\
& = \frac{\sqrt{\log n}}{2} - \log\log n +1.
\end{align*}
It is straightforward to verify that $x/2-2\cdot\log x+1$ attains its minimum at $x_0=4$. Thus, 
$x/2-2\cdot\log x+1 \geq 3-\log 16 > 0$, which implies that
\[ 
g\Big(n-\frac{n}{\log n} \Big(1+\frac{1}{\sqrt{\log n}}\Big)\Big) > 0.
\]
 
On the other hand, noting that $\frac{1}{1+x}>1-x$ for $x>0$, we obtain
\begin{align*}
g\Big(n-\frac{n}{\log n}\Big(1+\frac{1}{\log n}\Big)\Big)
& = \log\Big(\frac{n}{\log n}\Big(1+\frac{1}{\log n}\Big)\Big) 
- \frac{\log n}{1+(\log n)^{-1}} + 1 \\
& < \log\Big(\frac{n}{\log n}\Big(1+\frac{1}{\log n}\Big)\Big) 
- \Big(1-\frac{1}{\log n}\Big)\cdot\log n + 1 \\
& < -\log\log n + \frac{1}{\log n} + 2 \\
& < 2.125 - \log\log n \\
& < 0, 
\end{align*}
completing the proof of \eqref{eq:g-vs-zero}.

\section{Proof of Inequality \eqref{eq:k-lower-bound}} 
\label{appendix:k-lower-bound}

Here we give a proof of \eqref{eq:k-lower-bound}. Specifically, we show that
\[
\Big(1 {+} \frac{n}{\log n} \Big(1 {+} \frac{1}{\sqrt{\log n}}\Big)\Big)^{n-3-\frac{n}{\log n} \big(1+\frac{1}{\sqrt{\log n}}\big)}
> \Big(1 {+} \frac{n}{\log n} \Big(1 {+} \frac{1.1}{\sqrt{\log n}}\Big)\Big)^{n-1-\frac{n}{\log n} \big(1+\frac{1.1}{\sqrt{\log n}}\big)}. 
\]

\begin{proof}
For short, we set 
\[
A:=1+\frac{n}{\log n} \Big(1+\frac{1}{\sqrt{\log n}}\Big).
\]    
By simple algebra we find that
\begin{align*}
A\cdot\log A 
& > \Big(1+\frac{n}{\log n} \Big(1+\frac{1}{\sqrt{\log n}}\Big)\Big) (\log n-\log\log n) \\
& = n+\frac{n}{\sqrt{\log n}} - \frac{n\cdot\log\log n}{\log n} \Big(1+\frac{1}{\sqrt{\log n}}\Big)
+\log \frac{n}{\log n}.
\end{align*}
One can check that the function 
\[
h(x):=\frac{\log\log x}{\sqrt{\log x}} \Big(1+\frac{1}{\sqrt{\log x}}\Big)
\]
is decreasing in $x$ for $x>e^{e^2}$, we have $h(n)\leq h(5000)<0.986$. It follows that 
\begin{equation}\label{eq:A-logA}
A\cdot\log A > n+\frac{0.014 n}{\sqrt{\log n}} +\log \frac{n}{\log n}.
\end{equation}

For $x\geq 5000$, let 
\[
p(x):=\frac{0.014 x}{\sqrt{\log x}} +\log \frac{x}{\log x}
+\frac{x}{\log x}+\frac{1.1x}{(\log x)^{1.5}} - 40\sqrt{\log x}\cdot (\log x -1)+44.
\]
Taking the derivative of $p(x)$, we have
\begin{align*}
p'(x) 
= &\, \Big(\frac{0.014}{\sqrt{\log x}} - \frac{0.007}{(\log x)^{1.5}}\Big)
+ \Big(\frac{1}{\log x} - \frac{1}{(\log x)^2} - \frac{1.65}{(\log x)^{2.5}} - \frac{60\sqrt{\log x}}{x}\Big) \\
& +\Big(\frac{1.1}{(\log x)^{1.5}} - \frac{1}{x\cdot\log x}\Big) + \frac{1}{x} + \frac{20}{x\cdot\sqrt{\log x}}.
\end{align*}
Noting that the value in each bracket of right-hand side of the above equation greater than zero, we see $p'(x)>0$,
and hence $p(x)$ is increasing in $x$. Therefore, $p(n)\geq p(5000)>0$, which together with \eqref{eq:A-logA} gives
\begin{align*}
A\cdot \log A 
& > n - \frac{n}{\log n} - \frac{1.1n}{(\log n)^{1.5}} + 40\sqrt{\log n}\cdot (\log n -1) - 44 \\
& = \Big(n-\frac{n}{\log n} - \frac{1.1n}{(\log n)^{1.5}}\Big)
\Big(1+\frac{40\cdot (\log n)^{1.5}}{n}\Big) \\
& = \Big(n+1-\frac{0.1n}{(\log n)^{1.5}}-A\Big) \Big(1+\frac{40\cdot (\log n)^{1.5}}{n}\Big).
\end{align*}
Observe that $1+2x> (1-x)^{-1}$ for $0<x<1/2$. Then
\[
1+\frac{40\cdot (\log n)^{1.5}}{n} > \frac{1}{1-20\cdot (\log n)^{1.5}/n}.
\]
Combining these two inequalities above, we obtain 
\[
A\cdot\log A > \frac{1}{1-20\cdot (\log n)^{1.5}/n} \Big(n-\frac{0.1n}{(\log n)^{1.5}}-A\Big).
\]
Now, multiplying the above inequality by $\frac{0.1n (\log n)^{-1.5} -2 }{A}$, we deduce that
\begin{align*}
\Big(\frac{0.1n}{(\log n)^{1.5}}-2\Big) \cdot\log A 
& > \frac{0.1n}{A\cdot (\log n)^{1.5}} \Big(n-\frac{0.1n}{(\log n)^{1.5}} -A\Big) \\
& > \Big(n-\frac{0.1n}{(\log n)^{1.5}} -A\Big)\cdot\log\Big(1+\frac{0.1n}{A\cdot (\log n)^{1.5}}\Big),
\end{align*}
where the last inequality follows from the fact $x>\log (1+x)$ for $x>0$. Therefore,
\[
A^{\frac{0.1n}{(\log n)^{1.5}}-2} 
> \Big(1+\frac{0.1n}{A\cdot (\log n)^{1.5}}\Big)^{n-\frac{0.1n}{(\log n)^{1.5}} -A}.
\]

Finally, multiplying the above inequality by $A^{n - 0.1n(\log n)^{-1.5} - A}$, we arrive at the desired inequality.
\end{proof}

\section{Proof of Inequality \eqref{eq:k-upper-bound}}
\label{appendix:k-upper-bound}

Here we give a proof of \eqref{eq:k-upper-bound}. Specifically, we show that
\[
\Big(1+\frac{n}{\log n} \Big(1-\frac{1}{\log n}\Big)\Big)^{n-1-\frac{n}{\log n}\big(1-\frac{1}{\log n}\big)}
< \Big(\frac{n}{\log n} \Big(1+\frac{1}{\log n}\Big) -2\Big)^{n-\frac{n}{\log n} \big(1+\frac{1}{\log n}\big)}.
\]

\begin{proof}
For $x\geq 5000$, let 
\[
q(x)=\Big(\log\log x {+} \frac{\log\log x}{\log x} {-} \frac{1}{(\log x)^2} {-} 2\Big) x
{+} \frac{100\cdot\log x {-} 199}{99}\cdot\log x
{-} \frac{100\cdot(\log x)^3}{99}.
\]
The derivative of $q(x)$ relative to $x$ is
\[
q'(x) {=} \log\log x {-} 2 {+} \frac{200\cdot\log x}{99x} {+} \frac{\log\log x}{\log x}
{+} \frac{1}{\log x} {+} \frac{2}{(\log x)^3}
{-} \frac{\log\log x}{(\log x)^2} {-} \frac{199}{99x} {-} \frac{100\cdot (\log x)^2}{33x}.
\]
Observe that for $x\geq 5000$,
\[
\frac{\log\log x}{(\log x)^2} + \frac{199}{99x} + \frac{100\cdot (\log x)^2}{33x} 
< 0.1.
\]
It follows that $q'(x) > \log\log x -2.1 >0$, yielding that $q(x)$ is increasing in $x$. Therefore, $q(n)\geq q(5000) > 0$,
which, after some calculations, gives
\begin{equation}\label{eq:B-logB}
\frac{n}{\log n} \Big(1 {+} \frac{1}{\log n}\Big)\Big(\log n {+} \frac{1}{\log n} {-} \log\log n\Big)
{<} \Big(1 {-} \frac{(\log n)^2}{0.99n}\Big) \Big(n {-} \frac{n}{\log n}\Big(1 {-} \frac{1}{\log n}\Big) {-} 1\Big).
\end{equation}

To finish the proof, set for short
\[
B:= \frac{n}{\log n} \Big(1+\frac{1}{\log n}\Big) -2.
\]
It follows from \eqref{eq:B-logB} that
\begin{align*}
B\cdot\log B
& < \Big(1 {-} \frac{100\cdot (\log n)^2}{99n}\Big) \Big(n {-} \frac{n}{\log n}\Big(1 {-} \frac{1}{\log n}\Big) {-} 1\Big) \\
& < \Big(1-\frac{2}{2n/(\log n)^2-1}\Big) \Big(n {-} \frac{n}{\log n}\Big(1 {-} \frac{1}{\log n}\Big) {-} 1\Big),
\end{align*}
which, after some algebra, turns out to be equivalent to
\[
-\Big(\frac{2n}{(\log n)^2} -3\Big)
\Big(n {-} \frac{n}{\log n}\Big(1 {-} \frac{1}{\log n}\Big) {-} 1\Big)
< -\Big(\frac{2n}{(\log n)^2} -1\Big)\cdot B\cdot\log B.
\]
Noting that $\log (1-x) < -x$ for $x<1$, we obtain
\[
\Big(n - \frac{n}{\log n}\Big(1 - \frac{1}{\log n}\Big) - 1\Big)
\cdot\log \Big(1-\frac{2n/(\log n)^2 -3}{B} \Big)
< -\Big(\frac{2n}{(\log n)^2} -1\Big)\cdot \log B.
\]
This is equivalent to
\[
\Big(1-\frac{2n/(\log n)^2 -3}{B} \Big)^{n-\frac{n}{\log n}\big(1-\frac{1}{\log n}\big)-1}
< B^{-2n/(\log n)^2 +1}.
\]
Multiplying the above inequality by $B^{n-3+2n(\log n)^{-2} -B}$, we obtain the desired inequality.
\end{proof}

\end{document}